\documentclass[reqno,11pt]{amsart}
\usepackage{epsfig}
\usepackage{color}
\usepackage{times}
\usepackage{amsfonts,amsmath,amssymb}
\usepackage{mathrsfs}
\usepackage[latin1]{inputenc}
\usepackage[T1]{fontenc}
\numberwithin{equation}{section}
\newtheorem{theorem}{Theorem}[section]
\newtheorem{lemma}[theorem]{Lemma}

\allowdisplaybreaks

\newtheorem{proposition}[theorem]{Proposition}
\newtheorem{follow}[theorem]{Corollary}

\theoremstyle{definition}

\newcommand{\bel}{\begin{equation} \label}
\newcommand{\ee}{\end{equation}}

\newcommand{\pd}{\partial}

\newcommand{\C}{{\mathbb C}}
\newcommand{\R}{{\mathbb R}}
\newcommand{\N}{{\mathbb N}}

\newcommand{\Ds}{D_{x'}}
\newcommand{\Deltas}{\Delta_{x'}}

\newcommand{\nablas}{\nabla_{x'}}

\def\beq{\begin{equation}}
\def\eeq{\end{equation}}
\newcommand{\bea}{\begin{eqnarray}}
\newcommand{\eea}{\end{eqnarray}}
\newcommand{\beas}{\begin{eqnarray*}}
\newcommand{\eeas}{\end{eqnarray*}}
\newcommand{\Pre}[1]{\ensuremath{\mathrm{Re} \left( #1 \right)}}
\newcommand{\Pim}[1]{\ensuremath{\mathrm{Im} \left( #1 \right)}}
{

\newcommand{\abs}[1]{\left\lvert#1\right\rvert}
\newcommand{\norm}[1]{\left\lVert#1\right\rVert}
\newcommand{\FF}{\mathscr{F}_\gamma}
\newcommand{\p}{\partial}
\newcommand{\s}{\sigma}
\newcommand{\set}[1]{\left\{#1\right\}}
\newcommand{\para}[1]{\left(#1\right)}
\newcommand{\cro}[1]{\left[#1\right]}

\newcommand{\seq}[1]{\left<#1\right>}

\newcommand{\Nn}{\mathtt{N}}
\usepackage{graphicx}
\usepackage{color}

\begin{document}
\title[An inverse stability result for non compactly supported potentials]{An inverse stability result for non compactly supported potentials by one arbitrary lateral Neumann observation}
\author{M. Bellassoued}
\address{University of Carthage, Faculty of Sciences of Bizerte, Dep. of Mathematics, 7021 Jarzouna, Bizerte, Tunisie}
\email{mourad.bellassoued@fsb.rnu.tn}
\author{Y. Kian}
\address{Aix-Marseille Universit\'e, CNRS, CPT UMR 7332, 13288 Marseille, France \& Universit\'e de Toulon, CNRS, CPT UMR 7332, 83957 La Garde, France.}
\email{yavar.kian@univ-amu.fr}
\author{E. Soccorsi}
 \address{Aix-Marseille Universit\'e, CNRS, CPT UMR 7332, 13288 Marseille, France \& Universit\'e de Toulon, CNRS, CPT UMR 7332, 83957 La Garde, France.}
\email{eric.soccorsi@univ-amu.fr}

\begin{abstract}
In this paper we investigate the inverse problem of determining the time independent scalar potential of the dynamic Schr\"odinger equation in an infinite cylindrical domain, from partial measurement of the solution on the boundary.
Namely, if the potential is known in a neighborhood of the boundary of the spatial domain, we prove that it can be logarithmic stably determined in the whole waveguide from a single observation of the solution on any arbitrary strip-shaped subset of the boundary.
\end{abstract}

\maketitle


\section{Introduction}


In the present paper we seek global stability in the inverse problem of determining the (non necessarily compactly supported) zero-th order term (the so called electric potential) of the dynamical Schr\"odinger equation in an infinite cylindrical domain, from a single lateral observation of the solution over the entire time span. But in contrast to \cite{KPS2}, where the measurement is performed on a sub-boundary fulfilling the geometric control property expressed by Bardos, Lebeau and Rauch in \cite{[BLR]}, we allow in this paper that the Neumann data be taken on any infinitely extended strip-like subset of the lateral boundary, with positive Lebesgue measure.

\subsection{Inverse problem}
Let us make this statement a little bit more precise. We stick with the notations of \cite{KPS2}. Namely, $\omega$ is an open connected bounded domain in $\R^{n-1}$, $n \geq 3$, with smooth boundary $\p\omega$, and we consider the infinite straight cylinder $\Omega:=\omega\times\R$, in $\R^n$, with cross section $\omega$. Its boundary is denoted by $\Gamma:=\p\omega\times\R$. Given $T>0$, $p : \Omega \to \R$ and $u_0 : \Omega \to \R$, we consider the Schr\"odinger equation,
\bel{1.1}
-i \p_t u(x,t)-\Delta u(x,t)+p(x)u(x,t)=0,\ (x,t) \in \Omega \times (0,T),
\ee
associated with the initial data $u_0$,
\bel{1.2}
u(x,0)=u_0(x),\ x\in\Omega,
\ee
and the homogeneous Dirichlet boundary condition,
\bel{1.3}
u(x,t)=0,\ (x,t) \in \Gamma \times (0,T).
\ee
For suitable (real-valued) $u_0$ and $p$, and under appropriate compatibility conditions on these two functions, we denote by $u_p$ the unique $\mathcal{C}^0([0,T],H^1(\Omega))$-solution to the initial boundary value problem (abbreviated as IBVP in the sequel) \eqref{1.1}--\eqref{1.3}. Given an arbitrary relatively open subset $S_*$ of $\p \omega$, with positive Lebesgue measure, we aim for determining the unknown potential $p=p(x)$ from one Neumann observation of the function $u_p$ on $\Sigma_*:=\Gamma_* \times (0,T)$, where $\Gamma_*:=S_* \times \R$ is an infinitely extended strip designed on the boundary $\Gamma$ of the waveguide $\Omega$. We refer to \cite[Section 1]{KPS1} for both the relevance and the physical interpretation of the system modeled by \eqref{1.1}--\eqref{1.3} and the related inverse problem.

The uniqueness issue in the inverse problem examined in this paper is to know whether any two admissible potentials $p_j$, $j=1,2$, are equal, i.e. $p_1(x)=p_2(x)$ for a.e. $x \in \Omega$, if their observation data coincide, that is if the following identity holds true:
$$
\p_\nu u_{p_1}(x,t)=\p_\nu u_{p_2}(x,t),\ (x,t) \in \Sigma_*.
$$
Here $\nu=\nu(x)$, $x \in \Gamma$, denotes as usual the unit outward normal vector to $\Gamma$ and $\p_\nu u=\nabla u \cdot \nu$ stands for the normal derivative of $u$.
We shall give a positive answer to this question provided the two unknown functions $p_1$ and $p_2$ coincide in a neighborhood of the boundary $\Gamma$. This extra information imposed on the unknown zero-th order coefficient of \eqref{1.1} near $\Gamma$ is technically restrictive, but it is acceptable from a strict practical viewpoint upon admitting that the electric potential can be measured from outside the domain $\Omega$ in the vicinity of the boundary.

Actually, the above mentioned uniqueness result follows from a stronger statement claiming logarithmic stability in the determination of the potential $p$ from the observation of $\pd_\nu u_p$ on $\Sigma_*$. This amount to saying that $\| p_2-p_1 \|_{L^2(\Omega)}$ can be bounded from above in terms of (the logarithm of) a suitable norm of the trace of the function $\p_\nu u_{p_2}- \p_\nu u_{p_1}$ on $\Sigma_*$.
Such stability estimates play a key role in the analysis of ill-posed inverse problems (in the classical sense of \cite{[La]}), by suggesting regularization parameters and indicating the rate of convergence of the regularized solutions to the exact one.

The main achievement of this paper is that the Neumann data used in this stability estimate can be measured on any arbitrary unbounded strip-shaped subset $\Gamma_*=S_* \times \R$ of the whole boundary $\Gamma$.
The key idea of the proof is to combine the analysis carried out in
\cite{[Be1], [IY]}, which is based on a Carleman estimate specifically designed for the system under consideration, with
the Fourier-Bros-Iagolnitzer (abbreviated as FBI in the following) transformation used by Robbiano for
sharp unique continuation in \cite{[Ro]} (see also
\cite{[LR1],[Ph]}). Indeed we take advantage of the fact that the FBI transform of the (time derivative of the) solution to \eqref{1.1} satisfies a parabolic equation in the vicinity of the boundary $\Gamma$ in order to apply a Carleman parabolic estimate where no geometric condition is imposed on the control domain.

\subsection{Existing papers}
There is a wide mathematical literature on uniqueness
and stability issues in inverse coefficients problems of partial differential equations, see e.g. \cite{[Be2], [BeYa2],[IY], [PY]} and the references therein.
However, most of the known results on these two problems require that the corresponding Dirichlet or Neumann data be at least 
measured on a sufficiently large part $\Gamma_\sharp$ of the boundary $\Gamma$ of the spatial domain under consideration, if not on the whole boundary itself.

On the other hand, when $\Gamma_\sharp=\set{ x \in \Gamma,\ (x-x_0).\nu(x)\geq 0}$, where $x_0$ denotes a fixed point in the complement set of $\overline{\Omega}$, is the sub-boundary suggested by the geometric optics condition for the observability derived by Bardos, Lebeau and Rauch in \cite{[BLR]}, Baudouin and Puel \cite{[BP]} proved uniqueness and Lipschitz stability in the inverse problem of determining the electric potential of the Schr\"odinger equation from the observation of one Neumann data on $\Gamma_\sharp$.
In the present paper we claim logarithmic stability for arbitrarily small boundary parts $\Gamma_\sharp=\Gamma_*$ which do no necessarily comply with the geometric condition of Bardos, Lebeau and Rauch. Nevertheless this is at the expense of a stronger assumption on the potential, which is assumed to be known in a neighborhood of $\Gamma$.

In the framework this paper, we are dealing with a single observation of the solution. Uniqueness results for multidimensional inverse problems from a single observation of the solution were first derived by Bukhgeim and Klibanov \cite{[BK]} or Yamamoto \cite{[Y]} when $\Gamma_\sharp=\Gamma$, by means of suitable Carleman estimates. For the analysis of inverse coefficients problems with a finite number of observations, based on Carleman estimates, we also refer to, Bellassoued \cite{[Be1],[Be2]}, Bellassoued, Imanuvilov and Yamamoto \cite{[BeImYa]}, Bellassoued and Yamamoto \cite{[BeYa],[BeYa2]}, Bukhgeim \cite{[B]}, Bukhgeim, Cheng,
Isakov and Yamamoto \cite{[BCIV]}, Choulli and Yamamoto \cite{[CY1],[CY2]}, Imanuvilov and Yamamoto
\cite{[IY],[IY2],[IY3]}, Isakov \cite{[I1]}, Isakov and Yamamoto
\cite{[IsY]}, Kha\u idarov \cite{[KH1]}, Klibanov \cite{[KL]},
Klibanov and Yamamoto \cite{[KY]}, Puel and Yamamoto \cite{[PY]}, and
Yamamoto \cite{[Y]}. 

The stability issue in the inverse problem of determining the time-independent electric potential in the dynamic Schr\"odinger equation from a single boundary measurement was treated by 
Baudouin and Puel in \cite{[BP]} and by Mercado, Osses and Rosier in \cite{MeOsRo}. In these two papers the Neumann data is observed on a sub-boundary satisfying the geometric control condition of Bardos, Lebeau and Rauch. This condition was relaxed in \cite{BC} under the assumption that the potential is known near the boundary.

As for inverse problems for the non-stationary
Schr\"odinger equation by infinitely many boundary observations
(i.e. the Dirichlet-to-Neumann map, abbreviated as DN map in the sequel), we refer to e.g. Avdonin et al.
\cite{[ALP]}, where the real valued electric potential is retrieved from the partial knowledge of the DN map (the observation of the Neumann data is performed on a sub-part of $\Gamma$). 

In all the above mentioned papers the Schr\"odinger equation is defined in a bounded spatial domain. In the present work we rather investigate the problem of determining the scalar potential of the Schr\"odinger equation in an infinite cylindrical domain. Actually, there are only a few mathematical papers dealing with inverse coefficient problems in an unbounded domain. In \cite{LiUl}, Li and Uhlmann proved uniqueness in the determination of the compactly supported electric potential in an inifinite slab from partial DN map. In \cite{KPS1}, the compactly supported potential of the Schr\"odinger equation defined in an unbounded waveguide was Lipschitz stably retrieved from one measurement of the solution on a sub-boundary fulfilling the geometric control property of Bardos, Lebeau and Rauch. This result was extended to non compactly supported potentials in \cite{KPS2}, but Lipschitz stability degenerated to H\"older stability. Similar uniqueness results for non-compactly supported coefficients of the wave equation are derived by Rakesh in \cite{[Rakesh]} and Nakamura in \cite{[Nakamura]}, while the stability issue was treated by Kian in \cite{K}. 

\subsection{Main results}
We start by examining the direct problem associated with \eqref{1.1}--\eqref{1.3}. To this purpose we consider a fixed natural number $k \in \N^*:=\{ 1,2,\ldots \}$, and given $p_0 \in  W^{2(k-1),\infty}(\Omega)$ and $u_0 \in H^{2k}(\Omega)$, we set
$$
\mathrm{v}_0:=u_0\ \mbox{and}\ \mathrm{v}_j:=(-\Delta+p_0)\mathrm{v}_{j-1}\ \mbox{for}\ j=1,\ldots,k-1.
$$
We say that $u_0$ satisfies the $k$-th order compatibility conditions with respect to $p_0$ if the $k$ following identities
$$
\mathrm{v}_j (x)=0,\ x \in \Gamma,\ j=0,\cdots,k-1,
$$
hold simultaneously.
Evidently, if $u_0$ satisfies the $k$-th order compatibility conditions with respect to $p_0$ then it satisfies the $k$-th order compatibility conditions with respect to $p$ for any $p \in  W^{2(k-1),\infty}(\Omega)$ verifying $p=p_0$ in the vicinity of $\Gamma$.

Further, we introduce the set
$$ 
\mathcal{H}^k=\mathcal{H}^k(\Omega\times(0,T)):=
\bigcap_{j=0}^k\mathcal{C}^j([0,T],H^{2(k-j)}(\Omega)),
$$
where $H^k(\Omega)$ denotes the usual Sobolev space of order $k$ in $\Omega$. Endowed with the norm 
$$
\norm{u}_{\mathcal{H}^k}^2:=\sum_{j=0}^k\|\p_t^ju\|_{\mathcal{
C}^0([0,T],H^{2(k-j)}(\Omega))}^2,\  u\in \mathcal{H}^k,
$$
$\mathcal{H}^k$ is a Banach space, and we recall from \cite[Theorem 1.1]{KPS2} the following existence and uniqueness result for the system \eqref{1.1}--\eqref{1.3}.
\begin{proposition}
\label{L1.2}
For $k \in \N^*$ fixed, assume that $\pd \omega$ is $\mathcal{C}^{2k}$, and pick $p \in  W^{2(k-1),\infty}(\Omega)$ such that we have $\| p \|_{W^{2(k-1),\infty}(\Omega)} \leq M$ for some $M \geq 0$. Then for any $u_0 \in H^{2k}(\Omega)$ satisfying the $k$-th order compatibility conditions with respect to $p$, there exists a unique solution $u \in \mathcal{H}^k$ to the IBVP \eqref{1.1}--\eqref{1.3}. Moreover, the estimate
\bel{est-energie}
\| u \|_{\mathcal{H}^k} \leq C \| u_0 \|_{H^{2k}(\Omega)},
\ee
holds for some constant $C>0$, depending only on $\omega$, $T$, $k$ and $M$.
\end{proposition}

Put $N:=\left[ n/4 \right]+1$, where $\left[ s \right]$ stands for the integer part of $s \in \R$. Then, applying Proposition \ref{L1.2} with $k=N+1$, we get that
$u \in \mathcal{C}^1([0,T], H^{2 N}(\Omega))$ satisfies the estimate $\| u \|_{\mathcal{C}^1([0,T];H^{2 N}(\Omega))} \leq C \| u_0 \|_{H^{2(N+1)}(\Omega)}$. Further, as the embedding $H^{2 N}(\Omega) \hookrightarrow L^{\infty}(\Omega)$ is continuous, since $2 N > n \slash 2$, we deduce from Proposition \ref{L1.2} the following claim.
\begin{follow}
\label{cor-bounded}
Assume that the conditions of Proposition \ref{L1.2} are satisfied with $k=N+1$. Then there exists a positive constant $C$, depending only on $\omega$, $T$ and $M$, such that the solution $u$ to \eqref{1.1}--\eqref{1.3} satisfies the estimate:
$$
\| u \|_{\mathcal{C}^1([0,T],L^\infty(\Omega))} \leq C \| u_0 \|_{H^{2(N+1)}(\Omega)}.
$$
\end{follow}

Having seen this we turn now to introducing the inverse problem associated with \eqref{1.1}--\eqref{1.3}. We consider $p_0 \in  W^{2N,\infty}(\Omega;\R)$ and pick an open subset $\omega_0$ of $\omega$ such that $\pd \omega \subset \overline{\omega_0}$.
Given $b>0 $ and $d>0$, we aim to retrieve all functions $p : \Omega \to \R$ satisfying
\bel{1.8}
\Nn_{b,d}(p-p_0):=\| e^{b\seq{x_n}^d}(p-p_0)\|_{L^\infty(\Omega)}<\infty\ \mbox{and}\ p(x)=p_{0}(x)\ \mbox{for}\ x\in\Omega_0:=\omega_0\times\R.
\ee
Here and henceforth $\langle s \rangle $ is a short hand for $(1+s^2)^{1 \slash s}$, $s \in \R$. Notice that the assumption \eqref{1.8} is weaker than the compactness condition imposed in \cite[Theorem 1.1]{KPS1} on the support of the unknown part of $p$. 
Further, $M$ being an {\it a priori} fixed non-negative constant, we define the set of admissible potentials as
\bel{1.9}
\mathcal{A}(p_0,\omega_0):= \{ p \in W^{2 N,\infty}(\Omega);\ p=p_0\ \mbox{in}\ \Omega_0,\
\|p\|_{W^{2 N,\infty}(\Omega)}\leq M\ \mbox{and}\ \Nn_{b,d}(p-p_0)\leq M \}.
\ee
Last, we chose a relatively open subset $S_*$ of $\p \omega$, put $\Gamma_*:=S_* \times \R$ and introduce the norm
$$ \norm{\p_\nu u}_*:=\norm{\p_\nu
u}_{H^1(0,T;L^2(\Gamma_*))},\ u\in\mathcal{H}^k.
$$
The main result of this article is as follows.
\begin{theorem}
\label{thm-inv} 
Let the conditions of Proposition \ref{L1.2} be satisfied with $k=N+1$ and $p=p_0$.
Assume moreover that $u_0$ fulfills $\| u_0 \|_{H^{2(N+1)}(\Omega)} \leq M'$ for some constant $M'>0$, and that
\bel{1.11}
\exists \kappa>0,\ \exists d_0 \in (0, 2 d \slash 3),\ | u_0(x',x_n) | \geq \kappa \langle x_n \rangle^{-d_0 \slash 2},\ (x',x_n) \in \Omega \backslash \Omega_0.
\ee
For $p_j \in \mathcal{A}(p_0,\omega_0)$, $j=1,2$, we denote by $u_j$ the $\mathcal{H}^{N+1}$-solution to \eqref{1.1}--\eqref{1.3}, where $p_j$ is substituted for $p$. Then, for any $\epsilon\in (0,N \slash 2)$, there exists a constant $C=C(\omega, \omega_0, T, M, M', b, d, \epsilon)>0$, such that we have
\bel{1.12}
\norm{p_1-p_2}_{L^2(\Omega)}\leq
C \para{\norm{\p_\nu(u_1-u_2)}_*+\abs{\log \norm{\p_\nu(u_1-u_2)}_*}^{-1}}^{\epsilon}.
\ee
\end{theorem}

\subsection{Comments}

Thanks to the extra information $p_1 = p_2$ in the vicinity of $\Gamma$, the sharp unique continuation result by
Robbiano \cite{[Ro2]}, Robbiano and Zuily \cite{[RZ]} or Tataru
\cite{[Ta],[Ta2]}, entails $u_1 = u_2$ and $\nabla u_1 =
\nabla u_2$ on $\partial(\Omega\setminus\overline{\Omega_0}) \times
(0,T)$, provided $T>0$ is sufficiently large. Therefore the
method developed by Baudouin and Puel in \cite{[BP]} yields
uniqueness in the inverse problem under consideration. However, since we address the stability issue here, it is worth noticing that Theorem \ref{thm-inv} cannot be obtained by only combining the results of \cite{[IY2], [Ro2],[Ta]}. 

The technique carried out in this article may be applied, with appropriate modifications, to the determination of higher order
unknown coefficients of the Schr\"odinger equation, but in order to avoid the inappropriate expense of the size of this paper, we shall not go further into details about the treatment of this specific problem.

The analysis developed in this paper boils down to a new specifically designed Carleman estimate for the Schr\"odinger equation in
the cylindrical domain $\Omega \times (0,T)$, when the classical one is valid only
in level sets bounded by the weight function. For a general
treatment of Carleman estimates, we refer to H\"ormander \cite{[H]}, Isakov
\cite{[I1]}, Tataru \cite{[Ta]}, and also to Baudouin and Puel \cite{[BP]}, where Carleman estimates are derived in a direct pointwise manner. Due to the extra information $p_1=p_2$ in the vicinity of the boundary $\Gamma$, it is useless to discuss here the uniform Lopatinskii condition (see \cite{[Ta]} or \cite[Section 1.3]{[BeLR]}) or Carleman estimates with a reduced number of boundary traces.

We assume in \eqref{1.11} that $\abs{u(\cdot,0)}=\abs{u_0}>0$ in any subset of $\Omega$ where the electric potential is
retrieved. This is because the uniqueness of the potential is not known in general, without this specific assumption, even in
the case where the set $\{ x \in \Omega \setminus \Omega_0; u_0(x)=0 \}$
has zero Lebesgue measure. This non-degeneracy condition is
very restrictive but it is still an open question to know how it can be weakened in the context of inverse coefficients problems with a finite number of data observations.

Notice that in the framework of the Bukhgeim-Klibanov method in a bounded spatial domain $\Omega$, it is crucial that $|u_0|$ be bounded from below by a positive constant, uniformly in $\Omega$. But since $\Omega$ is infinitely extended here, such a statement is incompatible with the square integrability property satisfied by $u_0$ in $\Omega$. Therefore the usual non-degeneracy condition imposed on the initial condition function has to weakened into \eqref{1.11}. In the same spirit we point out that the derivation of a Carleman estimate in an unbounded domain such as $\Omega$ is not straightforward and does not directly follows from the corresponding known results in bounded domains.

The subset  $\{ x \in \Gamma,\ (x-x_0) \cdot \nu \geq 0 \}$, lying in the shadow of the
boundary $\Gamma$ viewed from a point $x_0$ taken in the complement set of $\overline{\Omega}$, satisfies the geometric control property introduced by
Bardos, Lebeau and Rauch in \cite{[BLR]}. This property is essentially a necessary and sufficient
condition for exact controllability and stabilization of wave equations. However, due to infinite speed of propagation in the Schr\"odinger equation, this
concept is not completely natural in the context of quantum systems. Nevertheless, Lebeau proved in
\cite{[L1]} that the above mentioned condition guarantees the boundary controllability of the Schr\"odinger
equation in $H^{-1}(\Omega)$ with $L^2(\Omega)$ boundary controls.

The remainder of the paper is organized as follows. In section 2, we establish a Carleman inequality for the Schr\"odinger equation and we state a stability estimate for unique continuation. These two results are needed in the proof of Theorem \ref{thm-inv}, which is given in Section 3. Finally, section 4 contains the proof of the logarithmic observation inequality stated in Section 2. 

\section{Preliminary Estimates}
\setcounter{equation}{0}
In this section we state two preliminary PDE estimates which are the main ingredients in the analysis of the inverse problem under study. To this end we introduce the following notations used throughout the entire text.
We consider three open subsets $\omega_j$, $j=1,2,3$, of $\omega_0$, obeying
\bel{2.1}
\omega_j \subsetneq \omega_{j-1}\ \mbox{and}\ \pd \omega \subset \pd \omega_j,
\ee
and we set
$$
\Omega_j:=( \omega \setminus \overline{\omega_j} )\times\R = \Omega \setminus ( \overline{\omega_j} \times\R )\ \mbox{and}\
Q_j:=\Omega_j \times (-T,T).
$$


\subsection{A Carleman estimate for the Schr\"odinger equation}
A Carleman estimate is a weighted $L^2$-norm inequality for a PDE solution. It is particularly useful for proving uniqueness in Cauchy problems or unique continuation results for PDEs with non-analytic
coefficients. Carleman estimates are also well adapted to energy estimation in PDEs (see e.g. Kazemi and Klibanov
\cite{[KK]} or Klibanov and Malinsky \cite{[KM]}). An alternative method for the derivation of energy inequalities, which is not applicable to the problem under consideration in this paper, can be found in \cite{[BLR]}.

It is Carleman who first derived in his pioneering paper \cite{[Carleman]}, a suitable inequality, which was later called a Carleman estimate, for proving uniqueness in a two-dimensional elliptic Cauchy problem. Since then, Carleman estimates have been extensively studied by numerous mathematicians. For the general theory of Carleman inequalities for PDEs with isotropic (resp. anisotropic) symbol and compactly supported functions, we refer to H\"ormander \cite{[H]} (resp. Isakov \cite{[I1]}).
For Carleman estimates with non-compactly supported functions, see Tataru
\cite{[Ta]}, Bellassoued \cite{[Be2]}, Fursikov
and Imanuvilov \cite{[FI]}, and Imanuvilov \cite{[Ima]}. Notice that a direct derivation of pointwise Carleman estimates for hyperbolic equations, which are applicable to non compactly supported functions, is available in Klibanov and Timonov's paper \cite{[KT]}.

Although Carleman estimates for Schr\"odinger operators in a bounded domain are rather classical (see e.g. 
\cite{[Albano], [BP], [Ta2]}), we seek in the context of this paper, a Carleman inequality for the operator
\bel{2.3}
P:= L+p\ \mbox{with}\ L:=-i\partial_t-\Delta,
\ee
acting in the infinite cylinder $\Omega$. We start by defining suitable weight functions. To this end 
we fix $x_0' \in \R^{n-1}\backslash\overline{\omega}$ and put
\bel{beta-tilde}
\tilde{\beta}(x'):=\abs{x'-x_0'}^2,\ x' \in \omega,
\ee
in such a way that $\tilde{\beta} \in \mathcal{C}^4(\overline{\omega})$. Here $| x' |$ denotes the Euclidian norm of $x' \in \R^{n-1}$.
Next, for every $x=(x',x_n) \in \Omega$, we set
\bel{2.4}
\beta(x):= \widetilde{\beta}(x')+K,\ {\rm where}\ K:= r \|\tilde{\beta}\|_{L^{\infty}(\omega)}\ {\rm for\
some}\ r>1, 
\ee
and we define two weight functions associated with the parameter $\lambda>0$: 
\bel{2.5}
\varphi(x,t):=\frac{e^{\lambda  \beta(x)}}{(T+t)(T-t)}\ {\rm and}\ \eta(x,t):=\frac{e^{2\lambda K} -
e^{\lambda \beta(x)}}{(T+t)(T-t)},\ (x,t) \in Q :=\Omega \times (-T,T). 
\ee
Finally, for all $s>0$, we denote by  $M_1$ (resp. $M_2$) the adjoint (resp. skew-adjoint) part of the operator $e^{-s \eta} L e^{s \eta}$, acting in $(\mathcal{C}_0^{\infty})'(Q)$, i.e.
\bel{2.6}
M_1 : = i \partial_t +
\Delta  + s^2 |\nabla \eta |^2\ {\rm and}\ M_2: = i s (\p_t \eta) + 2 s \nabla \eta \cdot \nabla  + s (\Delta \eta), 
\ee
where we recall that $L$ is the principal part of the operator $P$ given by \eqref{2.3}. 

Having said that, we now state the following global Carleman estimate for the operator $P$.

\begin{proposition}
\label{P2.2} 
Let $\beta$, $\varphi$ and $\eta$ be given by \eqref{beta-tilde}--\eqref{2.5}, and let the operators $M_j$, 
$j=1,2$, be defined by \eqref{2.6}.
Then there are two constants $s_0>0$ and $C>0$, depending only on $\omega$ and $T$,
such that the estimate
\bea
& & s  \|  e^{-s \eta}  \nabla_{x'} w   \|_{L^2(Q_2)}^2
+s^3  \| e^{-s \eta} w  \|_{L^2(Q_2)}^2  + \sum_{j=1,2} \|   M_j  e^{-s \eta} w \|_{L^2(Q_2)}^2  \nonumber \\
& \leq &  C  \left(   \|   e^{-s \eta}   P  w  \|_{L^2(Q)}^2 +  \|  e^{-s \eta}  \nabla_{x'} w   \|_{L^2(Q_3\setminus Q_2)}^2
+  \| e^{-s \eta} w  \|_{L^2(Q_3\setminus Q_2)}^2 \right), \label{2.7}
\eea
holds for all $s \geq s_0$ and any function $w \in L^2(-T,T;  H^1_0( \Omega ) )$ verifying $P w \in L^2(Q)$. 
\end{proposition}
\begin{proof}
The proof boils down to \cite[Proposition 3.3]{KPS1}, which provides two constants $s_0>0$ and $C>0$, both of them depending only on $\omega$, $T$ and $M$, such that we have 
\bel{2.8}
s  \|  e^{-s \eta}  \nabla_{x'} \tilde{w}   \|_{L^2(Q)}^2
+s^3  \| e^{-s \eta} \tilde{w}  \|_{L^2(Q)}^2  + \sum_{j=1,2} \|   M_j  e^{-s \eta} \tilde{w} \|_{L^2(Q)}^2 
 \leq   C  \|   e^{-s \eta}   P  \tilde{w}  \|_{L^2(Q)}^2 ,
\ee
for every $s \geq s_0$, and any $\tilde{w} \in L^2(-T,T; H^1_0( \Omega ) )$ obeying $P \tilde{w} \in L^2(Q)$ and $\partial_{\nu} \tilde{w} =0$ on $\Sigma$.
Next we pick a cut-off function $\chi \in \mathcal{C}^\infty(\R^{n-1},[0,1])$ satisfying
$$
\chi(x')=\left\{\begin{array}{ll} 1 & \mbox{if}\ x'\in \omega \setminus \omega_2 \\
0 & \mbox{if}\ x'\in \omega_3,
\end{array}
\right.
$$
and apply the estimate \eqref{2.8} to $\tilde{w}(x,t)=\chi(x') w(x,t)$. Using that $P\tilde{w}=\chi P w+ [P,\chi] w$, where
$[A,B]$ stands for the commutator of the operators $A$ and $B$, and taking into account that $[P,\chi]$ is a first order differential operator whose coefficients are supported in $\Omega_3 \setminus \Omega_2$, we obtain \eqref{2.7}.
\end{proof}

\subsection{Logarithmic stability of unique continuation}
The unique continuation of a solution to the Schr\"odinger equation \eqref{1.1} from lateral boundary data
on $\Gamma_*$ was proved by Phung in \cite{[Ph]}. The coming result claims stability for the same problem.

\begin{lemma}
\label{L2.1}
Let $p_j \in \mathcal{A}(p_0,\omega_0)$ for $j=1,2$, let $u_j$ be a
solution to the Schr\"odinger equation \eqref{1.1} where $p_j$ is substituted for $p$, and put $u:=u_1-u_2$. Then for each $T>0$
and all $\mu \in (0,1)$, we may find a constant $C=C(\omega,\omega_0,T,M,M',\mu)>0$, depending neither on $p_1$ nor on $p_2$, such that we have
\bel{2.11}
\| \pd_t u \|_{L^2( (\Omega_3\setminus\Omega_2) \times (0,T))}^2 + \| \nablas
\pd_t u \|_{L^2( ( \Omega_3\setminus\Omega_2) \times (0,T) )}^2 \leq C \para{\norm{\p_\nu
u}_*^2+\abs{\log \norm{\p_\nu u}_*}^{-1}}^{2\mu N}.
\ee
\end{lemma}
The proof of this result boils down to the analysis carried out by Robbiano in \cite{[Ro],[Ro2]} or
Phung in \cite{[Ph]}, by means of the FBI transformation. Since it is rather lengthy, we postpone it to Section 4.


\section{Proof of Theorem \ref{thm-inv}}
\label{sec-inv}
In this section we establish the stability estimate \eqref{1.12} by adapting the Bukhgeim-Klibanov method presented in \cite{[BK]}, to the context of the infinite waveguide $\Omega$.
The first step involves linearizing the system \eqref{1.1}--\eqref{1.3} and symmetrizing its solution with respect to the time variable $t$.

\subsection{Linearization and time symmetrization}
\label{sec-linearization}
With reference to the notations of Theorem \ref{thm-inv} we put $p:=p_2-p_1$ and notice that $u:=u_1-u_2$ is a $\mathcal{H}^{N+1}$-solution to the IBVP
\bel{3.1}
\left\{
\begin{array}{rcll}
-i \p_t u-\Delta u + p_1 u & = &  p u_2 & \mbox{in}\ \Omega \times (0,T) \\
  u(\cdot,0) & = & 0, & \mbox{in}\ \Omega\\ 
  u & = & 0, &  \mbox{in}\ \Gamma \times (0,T).
\end{array}
\right.
\ee
In particular we have $u \in \mathcal{C}^1([0,T];H^{2 N}(\Omega))$, hence upon differentiating \eqref{3.1} with respect to $t$, we get that $v:= \pd_t u \in {\mathcal H}^{N}$ is solution to the system
\bel{3.2}
\left\{
\begin{array}{rcll} 
-i \pd_t v-\Delta v + p_1 v &= & p \pd_t u_2 & \mbox{in}\ \Omega \times (0,T) \\
v(\cdot,0)  &= & i p u_0, & \mbox{in}\ \Omega \\
v & = & 0, & \mbox{on}\ \Gamma \times (0,T).
\end{array}
\right.
\ee
Further, putting $u_2(x,-t):=\overline{u_2(x,t)}$ for all $(x,t) \in \Omega \times (0,T]$ and bearing in mind that $u_0$ and $p$ are real-valued, we deduce from \eqref{3.2} that the function $v$, extended on $[-T,0) \times \Omega$ by setting $v(x,t):=-\overline{v(x,-t)}$, is the 
$\cap_{k=0}^{N} \mathcal{C}^k([-T,T],H^{2(N-k)}(\Omega))$-solution to the system
\bel{3.3}
\left\{
\begin{array}{rcll} 
-i \pd_t v-\Delta v + p_1 v & = & p \pd_t u_2 & \textrm{in}\ Q=\Omega  \times (-T,T) \\ 
  v(\cdot,0)  & = & i p u_0, & \textrm{in}\ \Omega \\ 
  v & = & 0, &  \textrm{on}\ \Sigma:=\Gamma \times (-T,T).
\end{array}
\right.
\ee
The second step in the derivation of \eqref{1.12} is to apply the global Carleman inequality of Proposition \ref{P2.2} to $v$ in order to establish Lemma \ref{L3.1} stated below.

\subsection{An {\it a priori} estimate}
We stick with notations of Subsection \ref{sec-linearization} and establish the following technical result, which is quite similar to \cite[Lemmas 3.3 \& 3.4]{KPS1} and \cite[Lemma 3.3]{KPS2}. Nevertheless we include the proof just for the convenience of the reader.

\begin{lemma}
\label{L3.1}
Let $v$ denote the $\mathcal{C}^1([-T,T],L^2(\Omega)) \cap \mathcal{C}^0([-T,T],H_0^1(\Omega) \cap H^2(\Omega))$-solution to \eqref{3.3}.
Then there exists a constant $C>0$, independent of $s$, such that we have
$$
\| e^{-s\eta(\cdot,0)} p u_0 \|_{L^2(\Omega)}^2 \leq C  s^{-3 \slash 2} \left(\| e^{-s\eta(\cdot,0)} p \pd_t u_2 \|_{L^2(Q)}^2 + 
\|  e^{-s \eta}  \nabla_{x'} v   \|_{L^2(Q_3 \setminus Q_2)}^2
+ \| e^{-s \eta} v  \|_{L^2(Q_3 \setminus Q_2)}^2 \right),
$$
uniformly in $s \in (0,+\infty)$. 
\end{lemma} 
\begin{proof}
Put $\phi(x,t):=e^{-s \eta(x,t)} \xi(x') v(x,t)$ for $(x,t) \in \Omega \times (-T,T)$, where $\xi \in\mathcal{C}_0^\infty(\omega)$ is a cut-off function satisfying
$$
\xi(x')=\left\{\begin{array}{ll}
1 & \mbox{if}\ x \in \omega \setminus \omega_1 \\
0 & \mbox{if}\ x \in \omega_2.
\end{array}
\right.
$$
In light of \eqref{2.4}-\eqref{2.5} we have $\lim\limits_{\substack{t \downarrow (-T) }} \eta(x,t)= +\infty$ for every $x \in \Omega$, and hence 
$\lim\limits_{\substack{t \downarrow (-T) }} \phi(x,t)= 0$.
As a consequence it holds true that
\begin{equation}\label{3.4}
\| \phi (\cdot,0) \|_{L^2(\Omega)}^2 =  \int_{\Omega \times (-T,0)} \pd_t | \phi |^2(x,t)  dx dt = 2 \Pre{\int_{\Omega \times (-T,0)} (\pd_t \phi) \overline{\phi}(x,t) dx dt} . 
\end{equation}
On the other hand, \eqref{2.6} and the Green formula yield
$$ \Pim {\int_{\Omega \times (-T,0)} (M_1 \phi) \overline{\phi}(x,t) dx dt} = \Pre {\int_{\Omega \times (-T,0)}   (\pd_t \phi) \overline{\phi}(x,t) dx dt}  +  R, $$
where
$$
R= \Pim {
\int_{\Omega \times (-T,0)} (\Delta \phi) \overline{\phi}(x,t) dx dt + s^2 \| (\nabla \eta) \phi  \|_{L^2( \Omega \times (-T,0) )}^2}= -\Pim{\| \nabla \phi \|_{L^2(  \Omega \times (-T,0) )}^2}=0.
$$
We deduce from this, \eqref{3.4} and the identity $\| \phi(\cdot,0) \|_{L^2(\Omega)} = \|  e^{-s \eta(\cdot,0 )}\xi v(\cdot,0 ) \|_{L^2(\Omega)}$, that
\beas
& & \|  e^{-s \eta(\cdot,0)}\xi v(\cdot,0 ) \|_{L^2(\Omega)}^2 = 2  \Pim { \int_{\Omega \times (-T,0)} (M_1 \phi) \overline{\phi}(x,t) dx dt} \leq  2 \| M_1 \phi  \|_{L^2(Q)} \|  \phi  \|_{L^2(Q)} \\
& \leq & s^{-3 \slash 2}  \left(  s\| e^{-s\eta}\nablas v \|_{L^2(Q_2)}^2 + s^3 \| e^{-s \eta} v \|_{L^2(Q_2)}^2 + \|  M_1  e^{-s \eta} v \|_{L^2(Q_2)}^2 \right).
\eeas
Finally, the desired result follows from this upon recalling \eqref{3.3}, applying Proposition \ref{P2.2} to $v$, and noticing from \eqref{2.4}-\eqref{2.5} that $\eta(x,t) \geq \eta(x,0)$ for all $(x,t) \in Q$.
\end{proof}

\subsection{End of the proof} 
For any fixed $y>0$ it follows from Lemma \ref{L3.1} that
\bel{3.6}
\| e^{-s\eta(\cdot,0)} p u_0 \|_{L^2(\omega \times (-y,y))}^2 \leq C s^{-3 \slash 2}\left( \| e^{-s\eta(\cdot,0)} p \pd_t u_2 \|_{L^2(Q)}^2 +  \varrho^2 \right),\ s >0,
\ee
where
$\varrho^2 :=  \| \nabla_{x'} v   \|_{L^2(Q_3 \setminus Q_2)}^2 + \| v  \|_{L^2(Q_3 \setminus Q_2)}^2$. 
Here and in the remaining part of the proof, $C$ denotes a generic positive constant that is independent of $s$. 

Notice from \eqref{1.11} and the vanishing of $p$ in $\Omega_0$ that the inequality $| (p u_0)(x)  | \geq \kappa \langle y \rangle ^{-d_0 \slash 2} \abs{p(x)}$ holds for every $x \in \omega \times (-y,y)$. Furthermore, we have $\| \pd_t u_2 \|_{L^{\infty}(Q)} \leq C$ by Corollary \ref{cor-bounded}, and $\eta(x,0) \geq 0$ for every $x \in \Omega$, by \eqref{2.4}-\eqref{2.5}, so we may deduce from \eqref{3.6} that
\bel{3.6b}
\left( \kappa^2 \langle y \rangle ^{-d_0}-C s^{-3 \slash 2} \right) \| e^{-s \eta(\cdot,0)} p \|_{L^2(\omega \times (-y,y))}^2 \leq 
C s^{-3 \slash 2} \left( \|  p \|_{L^2(\omega \times (\R \setminus (-y,y)))}^2 + \varrho^2 \right),\ s>0.
\ee
Thus, taking $s=(\kappa^2 \slash (2C) )^{-2 \slash 3} \langle y \rangle ^{2d_0\slash 3}$ in \eqref{3.6b} and recalling from \eqref{2.4}-\eqref{2.5} that $\| \eta(.,0) \|_{L^\infty(\Omega)} \leq e^{2 \lambda K} \slash T^2$, we obtain that
\bel{3.7}
\| p \|_{L^2(\omega \times (-y,y))}^2 \leq C  e^{C \langle y \rangle ^{2d_0 \slash 3}}\left( \| p \|_{L^2(\omega \times (\R \setminus (-y,y)))}^2 +\varrho^2 \right).
\ee
Moreover, as $\|  p \|_{L^2(\omega \times (\R \setminus (-y,y)))}^2  \leq C \| e^{-2b \langle \cdot \rangle^{d}} \|_{L^1(\R \setminus (-y,y))}$ from \eqref{1.8}, we have for any $\delta \in (0,b)$,
\bel{3.7b}
\|  p \|_{L^2(\omega \times (\R \setminus (-y,y)))}^2  
\leq C \| e^{-\delta \langle \cdot \rangle^{d} } \|_{L^1(\R)} e^{-(2 b-\delta) \langle y \rangle^{d}}
\leq C e^{-(2 b-\delta) \langle y \rangle^{d}}.
\ee
Putting this together with \eqref{3.7} we find that
\bel{3.8}
\| p \|_{L^2(\omega \times (-y,y))}^2 
\leq C   e^{C \langle y \rangle ^{2d_0 \slash 3}}  \left( e^{-(2b-\delta) \langle y \rangle ^{d}} +\varrho^2 \right).
\ee
Set $\varrho_\delta:=e^{-(2b-\delta)}$ and let us now examine the two cases $\varrho \in  (0,\varrho_\delta)$ and $\varrho \in [ \varrho_\delta , +\infty)$ separately. First, if $\varrho \in (0,\varrho_\delta)$ we take $y=y(\varrho) :=\left( \left(2\frac{\ln \varrho}{\ln \varrho_\delta}\right)^{2 \slash d}-1\right)^{1 \slash 2}$ in \eqref{3.8} so we have $\varrho^2=e^{-(2b-\delta) \langle y \rangle ^{d}}$, and consequently
\bel{3.8b}
\| p \|_{L^2(\omega \times (-y,y))}^2  \leq C  e^{C \langle y \rangle ^{2d_0 \slash 3}-(2b-\delta) \langle y \rangle ^{d}}.
\ee
Since $d>2d_0 \slash 3$ we have $\sup_{t \in (0,1)} e^{Ct^{2d_0 \slash 3}-\delta t^{d}}<+\infty$, whence \eqref{3.8b} yields
\bel{3.9}
\| p \|_{L^2(\omega \times (-y,y))}^2  \leq C \left( \sup_{t \in (1,+\infty)} e^{Ct^{2d_0 \slash 3}-\delta t^{d}} \right) e^{-2(b-\delta) \langle y \rangle ^{d}} \leq  C \varrho^{2 \theta},\ \varrho \in (0,\varrho_\delta),
\ee
with 
\bel{3.9b}
\theta:=\frac{b- \delta}{2b- \delta}\in (0,1 \slash 2).
\ee 
On the other hand, we have
$\| p \|_{L^2(\omega \times (\R \setminus (-y,y)))}^2 \leq C e^{-2(b-\delta) \langle y \rangle ^{d}} \leq  C \varrho^{2 \theta}$ for all $\varrho \in (0,\varrho_\delta)$,
by \eqref{3.7b}. This and \eqref{3.9} entail
\bel{3.11}
\| p\|_{L^2(\Omega)}^2 \leq C \varrho^{2 \theta},\ \varrho \in (0,\varrho_\delta).
\ee
In the case where $\varrho \in [\varrho_\delta,+\infty)$, we use the upper bound $\| p \|_{L^2(\Omega)}^2 \leq 
C \| e^{-2b \langle \cdot \rangle^{d}} \|_{L^1(\R)}$, arising from \eqref{1.8}, and obtain that
\bel{3.12}
\| p \|_{L^2(\Omega)}^2  \leq C \left( \| e^{-2b \langle \cdot \rangle^{d}} \|_{L^1(\R)} \slash \varrho_\delta^{2 \theta} \right) \varrho^{2 \theta} \leq C \varrho^{2 \theta},\ \varrho \in [ \varrho_\delta , +\infty ).
\ee
Now, recalling that $v=\pd_t u$ and applying Lemma \ref{L2.1} to $u$, we get
\bel{3.12b}
\varrho^2\leq C\para{\norm{\p_\nu
u}_*+\abs{\log \norm{\p_\nu u}_*}^{-1}}^{2\mu N},
\ee
for any arbitrary $\mu \in (0,1)$. Therefore, $\theta$ being any real number in $(0, 1 \slash 2)$, according to \eqref{3.9b} and since $\delta$ is arbitrary in $(0,b)$, the estimate \eqref{1.12} follows from \eqref{3.11}--\eqref{3.12b} upon taking $\epsilon=\theta \mu N$.

\section{Logarithmic observability inequality: Proof of Lemma \ref{L2.1}}
\setcounter{equation}{0}

In this section we prove the logarithmic observability inequality stated in Lemma \ref{L2.1}. 

Prior to doing that we recall for further reference from the energy inequality \eqref{est-energie} with $k=N+1$, that for any $p_j \in\mathcal{A}(\omega_0,M)$, $j=1,2$, the solution $v=\partial_t(u_1-u_2)$ to the IBVP \eqref{3.3} satisfies the estimate
\bel{4.3}
\norm{v}_{\mathcal{C}^{N}([-T,T],L^2(\Omega))} + \norm{v}_{\mathcal{C}^{N-1}([-T,T],H^2(\Omega))} \leq
2 C \| u_0 \|_{H^{2(N+1)}(\Omega)},
\ee
where the positive constant $C=C(\omega,\omega_0,T,M,M')>0$ is the same as in \eqref{est-energie}. 

\subsection{A parabolic Carleman estimate for unbounded cylindrical domains}
\label{sec-parabolicCE}
In this subsection we state a parabolic Carleman estimate for the Schr\"odinger equation in unbounded cylindrical domains, which
is needed in the proof of Lemma \ref{L2.1}. To do that, we start by introducing the set
$$ S_\sharp := \partial \omega_0 \setminus \partial \omega\ \mbox{and}\ \Gamma_\sharp := S_\sharp \times \R, $$
and we assume without loss of generality (upon possibly smoothening $\pd \omega_0$ by enlarging $\omega_0$), that $S_\sharp$ is $\mathcal{C}^2$. Then, with reference to \cite[Lemma 2.3]{[IY2]} and its proof, we pick a function $\psi_0 \in \mathcal{C}^2(\overline{\omega}_0)$, obeying the four following conditions:
\bea
\psi_0(x')>0, x' \in \omega_0& \mbox{and} & \abs{\nabla\psi_0(x')}>0,\ x'\in \overline{\omega}_0, \label{iy1} \\
\psi_0(x')=0, x' \in S_\sharp& \mbox{and} & \pd_\nu \psi_0(x') \leq 0,\ x'\in \pd \omega_0 \setminus S_*. \label{iy2}
\eea

Next we put $\ell(\tau):=(1-\tau)(1+\tau)$ for each $\tau \in (-1,1)$, and introduce the two functions
\bel{4.8}
\varphi_0(x',\tau):=\frac{e^{\lambda(\psi_0(x')+a)}}{\ell(\tau)},\ x \in \omega_0,\ \tau\in(-1,1),
\ee
and
\bel{4.9}
\alpha(x',\tau):=\frac{e^{\lambda(\psi_0(x')+a)}-e^{\lambda(\norm{\psi_0}_{L^{\infty}(\omega_0)}+b)}}{\ell(\tau)},\ x\in\omega_0,\ \tau\in(-1,1),
\ee
where $\lambda \in (0,+\infty)$ is a fixed parameter, $\psi_0$ is the function defined by \eqref{iy1}-\eqref{iy2}, and
$$
\norm{\psi_0}_{L^{\infty}(\omega_0)}<a<b<2a-\norm{\psi_0}_{L^{\infty}(\omega_0)}.
$$
Further, in connection with the Schr\"odinger operator $P$ defined in \eqref{2.3}, we consider the formal
parabolic operator in $\Omega_0=\omega_0\times\R$, associated with some fixed parameter $h \in (0,1)$,
\bel{4.10}
\mathcal{L}_h:=h^{-1} \p_\tau-\Delta + p_1.
\ee
We are now in position to state the following Carleman estimate for the operator $\mathcal{L}_h$.

\begin{lemma}
\label{L4.2}
Let $\varphi_0$ and $\alpha$ be defined by \eqref{4.8}-\eqref{4.9}, and for $h \in (0,1)$ fixed, let $\mathcal{L}_h$ be defined by \eqref{4.10}. Then we may find three positive constants $\lambda_0$, $\s_0$ and
$C_0$, such that for every $\lambda \geq \lambda_0$ and $\s \geq \s_0 \slash h$, the estimate
\bea
& & \s \| e^{\s \alpha} \nablas w \|_{L^2(\Omega_0 \times (-1,1))}^2 + \s^3 \| e^{\s \alpha} w \|_{L^2(\Omega_0 \times (-1,1))}^2 \nonumber \\
& \leq & 
C_0 \left( \| e^{\s \alpha}  \mathcal{L}_h w \|_{L^2(\Omega_0 \times (-1,1))}^2 + 
\s \| \varphi_0^{1 \slash 2} e^{\s \alpha} \pd_\nu w \|_{L^2(\Gamma_* \times (-1,1))}^2 \right), \label{4.11}
\eea
holds for $w \in L^2(-1,1;H_0^1(\Omega_0))$ verifying $\mathcal{L}_h w \in L^2(-1,1;L^2(\Omega_0))$ and $\p_\nu w \in
L^2(-1,1;L^2(\Gamma_*))$. Here the constant $C_0>0$ depends continuously on $\lambda$,
$M$, $M'$ and $h$, but is independent of $\s$.
\end{lemma}
We stress out that a result similar to Lemma \ref{L4.2} can be found in \cite[Lemma 2.4]{[IY2]} (see also \cite{[FZ], [FI]}) in the context of bounded spatial domains.

The dependence of the various constants appearing in \eqref{4.11}, with respect to the parameter $h \in (0,1)$, is made precise in the derivation of Lemma \ref{L4.2}, which is given in Appendix A.

\subsection{A connection between Schr\"odinger and parabolic equations}

As pointed out by Lebeau and Robbiano \cite{[LR1]}, Robbiano \cite{[Ro2]}, Robbiano and Zuily
\cite{[RZ]} and Phung \cite{[Ph]}, connections between solutions of different types of PDEs may be useful for examining the controllability of numerous Cauchy problems. In this subsection we prove that the FBI transform of $\chi v$, where $v$ is the solution to \eqref{3.3} and $\chi=\chi(x')$ is a suitable cut-off function that will be made precise below, is solution to a parabolic Cauchy problem in $\Omega_0=\omega_0 \times \R$.

Prior to doing that we introduce the FBI transform as defined by Lebeau and Robbiano in \cite{[LR1]}.
To this purpose we fix $\mu \in (0,1)$ and choose $m \in \N^*$ so large that 
\bel{def-m}
2m \geq N\ \mbox{and}\ \rho:=1 -\frac{1}{2m}>\mu.
\ee
Then for any $\gamma \in (1,+\infty)$, the function
\bel{b0}
F_{\gamma}(z):=\frac{1}{2\pi}\int_{\mathbb{R}} e^{iz\eta}e^{-(\eta/\gamma^{\rho})^{2m}}d\eta,\ z\in\C,
\ee
is holomorphic in $\mathbb{C}$, and there exist four positive constants $C_j$, $j=1,2,3,4$, none of them depending on $\gamma$, such that we have
\bel{b1}
\left\vert F_{\gamma}(z)\right\vert \leq C_{1}\gamma^{\rho}e^{C_{2}\gamma\left\vert \text{Im}z\right\vert^{1 \slash \rho}},\ z\in\mathbb{C},
\ee
and
\bel{b1b}
\left\vert F_{\gamma}(z)\right\vert \leq C_{1}\gamma^{\rho}e^{-C_{3}\gamma\left\vert \textrm{Re}z\right\vert ^{1/\rho}},\ z \in \set{z\in\C,\ |\mathrm{Im}z|\leq C_4|\mathrm{Re}z|}.
\ee
Given $T_0 \in ( T \slash 3 , +\infty)$, we consider a cut-off function $\theta \in \mathcal{C}_0^\infty(\R)$ obeying
\bel{b2}
\theta(\eta)=\left\{\begin{array}{ll} 1 & \mbox{if}\ \abs{\eta}\leq 2T_0 \\
0 &  \mbox{if}\ \abs{\eta}\geq 3T_0,
\end{array}
\right.
\ee
and we define the partial FBI transform of $w \in \mathcal{S}(\R^{n+1})$ by
\bel{b3}
w_{\gamma,t}(x,\tau):=\mathscr{F}_\gamma w(x,z)=\int_{\mathbb{R}}F_{\gamma}(z-\eta) \theta(\eta) w(x,h\eta)d\eta,\ z=t-i\tau,
\ee
for all $t \in (-T_0,T_0)$, $\tau \in (-1,1)$, $\gamma \in (1,+\infty)$ and $x \in \R^n$, where $h:=T \slash (3T_0)$. 

Next, taking into account that $\omega_1 \subset \omega_0$, by \eqref{2.1}, we deduce
from the continuity of the function $\psi_0$ introduced in Subsection \ref{sec-parabolicCE}, and from the first part of \eqref{iy1}, that there exists a constant $\beta_0>0$ such that
\bel{4.6}
\psi_0(x') \geq 2 \beta_0,\ x' \in \omega_2 \setminus \omega_3.
\ee
Moreover, due to the vanishing of $\psi_0$ on $S_\sharp$, imposed by the first claim of \eqref{iy2}, we may find
a subset $\omega^\sharp \subset \omega_0 \setminus \overline{\omega_1}$ such that
\bel{4.7}
S_\sharp \subset \overline{\omega^\sharp}\ \mbox{and}\ \psi_0(x') \leq \beta_0\ \mbox{for}\ x'\in\omega^\sharp.
\ee
Let us now pick $\widetilde{\omega}^\sharp \subset \omega^\sharp$ such that $S_\sharp \subset \overline{\tilde{\omega}^\sharp}$, and introduce a function $\chi\in \mathcal{C}^\infty(\R^{n-1},[0,1])$ satisfying
\bel{4.12}
\chi(x')=\left\{
\begin{array}{ll}
1 & \mbox{if}\ x'\in \omega_0 \setminus \omega^\sharp \\
0 & \mbox{if}\ x'\in\widetilde{\omega}^\sharp.
\end{array}
\right.
\ee
Thus, bearing in mind that $p=p_1-p_2$ vanishes in $\Omega_0$ and that $v$ is the solution to \eqref{3.3}, we easily find that the function
$w(x,t):=\chi(x') v(x,t)$ satisfies the IBVP
\bel{4.13}
\left\{
\begin{array}{rcll}
-i\partial_t w-\Delta w+ p_1 w & = & -\cro{\Delta,\chi}v,  & \mbox{in}\ Q_0:=\Omega_0 \times (-T,T)\\
w(0,\cdot) & = & 0, & \textrm{in}\ \Omega_0 \\
w & = & 0 & \textrm{on}\ \Sigma_0:= \p \Omega_0 \times (-T,T).
\end{array}
\right.
\ee
Moreover, we deduce from \eqref{4.3} that
\bel{4.14}
\norm{w}_{{\mathcal C}^{N}([-T,T],L^2(\Omega))} + \norm{w}_{{\mathcal C}^{N-1}([-T,T],H^2(\Omega))} \leq C \| u_0 \|_{H^{2(N+1)}(\Omega)},
\ee
where $C$ denotes a generic positive constant that is independent of $\gamma$.
From this, \eqref{b1} and \eqref{b3}, we get two positive constants $C=C(\omega,\omega_0,T,T_0,M,M')$ and $\delta_1$, the last one being independent of $T_0$, such that the estimate
$$
\| w_{\gamma,t} \|_{L^2(\Omega_0 \times (-1,1))}^2 + \| \nabla w_{\gamma,t}\|_{L^2(\Omega_0 \times (-1,1))}^2 \leq C e^{\delta_1\gamma},
$$
holds uniformly in $t \in (-T_0,T_0)$ and $\gamma \in (1,+\infty)$. 

We turn now to establishing that $w_{\gamma,t}$ is solution to a parabolic Cauchy problem in $\Omega_0$, we shall make precise below. To do that we derive from \eqref{b3} upon integrating by parts that
$$
h^{-1}\p_\tau w_{\gamma,t}(x,\tau)=-i\FF(\p_t w)(x,z)-ih^{-1}\int_\R F_\gamma(z-\eta)\theta'(\eta) w(x,h \eta)d\eta,\ z=t - i \tau.
$$
Next, as we have $\Delta w_{\gamma,t}(x,\tau)=\FF(\Delta w)(x,z)$ by direct calculation, we get upon applying the FBI transform $\FF$ to \eqref{4.13} and remembering \eqref{4.10}, that
\bel{4.20}
\left\{ \begin{array}{rcll}
\mathcal{L}_h w_{\gamma,t}(x,\tau) & = & A_{\gamma,t}(x,\tau)+B_{\gamma,t}(x,\tau), & (x , \tau) \in \Omega_0 \times (-1,1), \\
w_{\gamma,t}(x,\tau) & = & 0, & (x , \tau ) \in \pd \Omega_0 \times (-1,1),
\end{array} \right.
\ee
where
\bel{4.21}
A_{\gamma,t}(x,\tau):=-\int_\R F_\gamma(z-\eta)\theta(\eta)\cro{\Deltas,\chi}v(x,h\eta)d\eta = -\cro{\Deltas,\chi} v_{\gamma,t}(x,\tau),
\ee
and
\bel{4.22}
B_{\gamma,t}(x,\tau):=-ih^{-1} \int_\R F_{\gamma}(z-\eta)\theta'(\eta) w(x,\eta h)d\eta,\ z=t - i \tau.
\ee

The next step of the proof is to apply the parabolic Carleman estimate of Lemma \ref{L4.2} to the solution $w_{\gamma,t}$ of \eqref{4.20} in order to derive the coming result.
\begin{lemma}
\label{L4.4}
There exists $\varepsilon \in (0,1)$, $\delta_j>0$ for $j=2,3$, and $\gamma_0>0$, such that
any solution $w_{\gamma,t}$ to \eqref{4.20} satisfies the estimate
$$
\| w_{\gamma,t} \|_{L^2((\Omega_3 \setminus \Omega_2) \times (-\varepsilon,\varepsilon))}^2 + \| \nablas w_{\gamma,t} \|_{L^2((\Omega_3 \setminus \Omega_2) \times (-\varepsilon,\varepsilon))}^2 
\leq C \left( e^{-\delta_2 \gamma} + e^{\delta_3 \gamma} \| \p_\nu w_{\gamma,t} \|_{L^2(\Gamma_* \times (-1,1))}^2 \right),
$$
uniformly in $t \in (-T_0,T_0)$ and $\gamma \in [ \gamma_0,+\infty)$.
\end{lemma}
\begin{proof}
Fix $\gamma \in (1,+\infty)$ and $t \in (-T_0,T_0)$. In light of \eqref{4.20}, we apply the Carleman estimate of Lemma \ref{L4.2} to $w_{\gamma,t}$, and find for every $\s \in [\sigma_0 \slash h,+\infty)$ that
\bea
& & \s \| e^{\s \alpha} \nablas w_{\gamma,t} \|_{L^2(\Omega_0 \times (-1,1))}^2 + \s^3 \| e^{\s \alpha} w_{\gamma,t} \|_{L^2(\Omega_0 \times (-1,1))}^2 \nonumber \\
& \leq & 
C_0 \left( \| e^{\s \alpha}  A_{\gamma,t} \|_{L^2(\Omega_0 \times (-1,1))}^2 + \| e^{\s \alpha}  B_{\gamma,t} \|_{L^2(\Omega_0 \times (-1,1))}^2 +
\s \| \varphi_0^{1 \slash 2} e^{\s \alpha} \pd_\nu w_{\gamma,t} \|_{L^2(\Gamma_* \times (-1,1))}^2 \right). \label{4.27}
\eea
Further, we notice from \eqref{4.12} that $A_{\gamma,t}(\cdot,\tau)$ is supported in $\Omega_\sharp:= \omega_\sharp \times \R$ for every $\tau \in (-1,1)$, and from \eqref{4.9} and \eqref{4.7} that
$\alpha(x',\tau) \leq (-\mu_1)$ for all $(x',\tau) \in \omega^\sharp \times (-1,1)$, with
$\mu_1:=e^{\lambda(\norm{\psi_0}_\infty+b)}-e^{\lambda(\beta_0+a)}>0$. As a consequence we have
\bel{4.28}
\| e^{\s\alpha} A_{\gamma,t} \|_{L^2(\Omega_0 \times (-1,1))}^2  \leq e^{-2 \mu_1 \s} \|  A_{\gamma,t} \|_{L^2(\Omega_\sharp \times (-1,1))}^2 \leq e^{-2 \mu_1\s} \|  A_{\gamma,t} \|_{L^2(\Omega_0 \times (-1,1))}^2.
\ee
The next step is to chose $\varepsilon \in (0,1)$ so small that 
$\mu_2 ( e^{\lambda(\norm{\psi_0}_\infty+b)}-e^{\lambda(2\beta_0+a)} ) \slash \ell(\varepsilon) < \mu_1$.
Then, bearing in mind that $\ell(\tau) \geq \ell(\varepsilon) > 0$ for each $\tau \in (-\varepsilon, \varepsilon)$, we see from
\eqref{4.6} that
$\alpha(x',\tau) \geq (-\mu_2)$ for every $(x',\tau) \in (\omega_2 \setminus \omega_3) \times (-\varepsilon,\varepsilon)$.
This entails that
\beas
& & e^{-2 \mu_2 \s} \left( \| \nablas w_{\gamma,t} \|_{L^2(( \Omega_3 \setminus \Omega_2) \times (-\varepsilon,\varepsilon))}^2 
+ \| w_{\gamma,t} \|_{L^2( ( \Omega_3 \setminus \Omega_2) \times (-\varepsilon,\varepsilon) )}^2 \right) \nonumber \\
& \leq & 
\s \| e^{\s \alpha } \nablas w_{\gamma,t} \|_{L^2( \Omega_0 \times (-1,1))}^2
+ \s^3 \| e^{\s \alpha} w_{\gamma,t} \|_{L^2( ( \Omega_0 \times (-1,1) )}^2,\ \s \in [1,+\infty). \label{4.31}
\eeas
Setting $\mu:=\mu_1-\mu_2$, it follows from this and \eqref{4.27}-\eqref{4.28} that
\bea
& & \| \nablas w_{\gamma,t} \|_{L^2(( \Omega_3 \setminus \Omega_2) \times (-\varepsilon,\varepsilon))}^2 
+ \| w_{\gamma,t} \|_{L^2( ( \Omega_3 \setminus \Omega_2) \times (-\varepsilon,\varepsilon) )}^2 \nonumber \\
& \leq & C \left( 
e^{-2 \mu \s} \| A_{\gamma,t} \|_{L^2(\Omega_0 \times (-1,1))}^2 + 
e^{2\mu_2\s} \| B_{\gamma,t} \|_{L^2(\Omega_0 \times (-1,1))}^2 + 
\s e^{2\mu_2\s} \| e^{\s \alpha} \varphi_0^{1 \slash 2} \pd_\nu w_{\gamma,t} \|_{L^2(\Gamma_* \times (-1,1))}^2 \right), \label{4.32}
\eea
whenever $\sigma \in [\sigma_0 \slash h,+\infty)$. Here we assumed upon possibly substituting $\max(1, \sigma_0 \slash h)$ for $\sigma_0$, that $\sigma_0 \geq h$. 

In view of \eqref{b1} and \eqref{4.21}, the first term in the right hand side of \eqref{4.32} can be treated with the energy estimate 
$\| v \|_{\mathcal{C}^0([-T,T],H^1(\Omega))} \leq 2C \| u_0 \|_{H^{2(N+1)}(\Omega)}$, 
arising from \eqref{4.3}: We get a constant $\delta'>0$, independent of $T_0$ and $\gamma$, such that
\bel{4.24}
\| A_{\gamma,t} \|_{L^2(\Omega_0 \times (-1,1))}^2 \leq Ce^{\delta'\gamma},\ t \in (-T_0,T_0).
\ee
For the second term, we take into account the vanishing of $\theta'$ in the interval $(-2T_0,2T_0)$, imposed by \eqref{b2}, and deduce from \eqref{b1b} and \eqref{4.22} that
\bel{4.23}
\| B_{\gamma,t} \|_{L^2(\Omega_0 \times (-1,1))}^2 \leq
C e^{-\tilde{\delta} T_0^{1/\rho} \gamma},\ t \in (-T_0,T_0),
\ee
for some constant $\tilde{\delta}>0$ depending neither on $T_0$ nor on $\gamma$. Here we used the estimate $\| w \|_{{\mathcal C}^0([-T,T],L^2(\Omega))} \leq C \| u_0 \|_{H^{2(N+1)}(\Omega)}$ arising from \eqref{4.14}.

Last we notice from \eqref{iy2}--\eqref{4.9} that $\varphi_0^{1 \slash 2} e^{\s\alpha}$ is bounded on $S_* \times(-1,1)$, and
then deduce from \eqref{4.32}--\eqref{4.23} that 
\bea
& & \| \nablas w_{\gamma,t} \|_{L^2(( \Omega_3 \setminus \Omega_2) \times (-\varepsilon,\varepsilon))}^2 
+ \| w_{\gamma,t} \|_{L^2( ( \Omega_3 \setminus \Omega_2) \times (-\varepsilon,\varepsilon) )}^2 \nonumber \\
& \leq & C \left( e^{-2\mu \s + \delta' \gamma} + e^{2 \mu_2\s -\tilde{\delta} T_0^{1/\rho}\gamma} + \s e^{2\mu_2 \s} \| \pd_\nu w_{\gamma,t} \|_{L^2(\Gamma_* \times (-1,1))}^2 \right),\ \s \in [ \s_0 \slash h , +\infty).
\label{4.33}
\eea
Now, set $\gamma_0:=\max \left( 1, 3\s_0 \slash T \right)$ and for $\gamma \in [\gamma_0,+\infty)$, take $\s:=T_0 \gamma \geq \s_0 \slash h$ in \eqref{4.33}.
As the sum of the two first terms in the right hand side of \eqref{4.33} is majorized by
$e^{(-2\mu T_0+ \delta') \gamma}+e^{(2\mu_2
T_0-\tilde{\delta} T_0^{1/\rho})\gamma} \leq Ce^{-\delta_2\gamma}$ upon taking $T_0$ sufficiently large, since $1 \slash \rho >1$, we end up getting for all $t \in (-T_0,T_0)$ that
\beas
& & \| \nablas w_{\gamma,t} \|_{L^2(( \Omega_3 \setminus \Omega_2) \times (-\varepsilon,\varepsilon))}^2 
+ \| w_{\gamma,t} \|_{L^2( ( \Omega_3 \setminus \Omega_2) \times (-\varepsilon,\varepsilon) )}^2 \\
& \leq &
C \left( e^{-\delta_2\gamma} +e^{\delta_3 \gamma} \| \pd_\nu w_{\gamma,t} \|_{L^2(\Gamma_* \times (-1,1))}^2 \right),\
\gamma \in [\gamma_0,+\infty),
\eeas
which entails the desired result.
\end{proof}
\subsection{Completion of the proof}

Set $w_{\gamma}(x,t):=w_{\gamma,t}(x,0)$ and recall from \eqref{b3} that we have
\bel{b5}
w_{\gamma}(x,t)=(F_\gamma \ast (\theta \widetilde{w})(x,\cdot))(t),
\ee
for all $\gamma \in [\gamma_0,+\infty)$, $x \in \R^n$ and $t \in (-T_0,T_0)$,
where $F_\gamma$ is defined in \eqref{b0} and $\widetilde{w}(x,\eta):=w(x,h\eta)$.

Let us first deduce from Lemma \ref{L4.4} the following estimate on $w_{\gamma}$.

\begin{lemma}
\label{L4.5}
There exist
two positive constants $\delta_j$, $j=4,5$, 
such that the estimate
\beas
& & \| \nablas w_{\gamma} \|_{L^2(( \Omega_3 \setminus \Omega_2) \times (-T_0 \slash 2,T_0 \slash 2))}^2 
+ \| w_{\gamma} \|_{L^2( ( \Omega_3 \setminus \Omega_2) \times (-T_0 \slash 2,T_0 \slash 2) )}^2 \\
& \leq &
C \left( e^{-\delta_4 \gamma} + e^{\delta_5 \gamma} \| \pd_\nu \tilde{w} \|_{L^2(\Gamma_* \times (-3T_0,3 T_0))}^2 \right),
\eeas
holds for all $\gamma  \in [\gamma_0,+\infty)$ and for $T_0$ sufficiently large.
\end{lemma}
\begin{proof}
Let $\varepsilon \in (0,1)$ be given by Lemma \ref{L4.4} and fix $\kappa \in [T_0 - \varepsilon , T_0 + \varepsilon ]$. 
We assume (without restricting the generality of the reasoning) that $\varepsilon <T_0 \slash 2$.
Since $w_{\gamma}(x,z) := w_{\gamma,\Pre{z}}(x,\Pim{z})$ is analytic in $z \in \{ \zeta \in \C,\ \Pre{\zeta} \in (-T_0,T_0),\ \Pim{\zeta} \in (-1,1) \}$ for every fixed $x \in \omega_2 \setminus \omega_3$, the Cauchy formula yields
$$
w_\gamma(x,\kappa) = \frac{1}{2i\pi}\int_{\vert
z-\kappa\vert=\varrho} \frac{w_\gamma(x,z)}{z-\kappa}dz
= \frac{1}{2\pi} \int_0^{2\pi} w_\gamma(x,\kappa+\varrho e^{i\phi}) d \phi,\ \varrho \in (0,\varepsilon).
$$
Therefore we have $\abs{w_\gamma(x,\kappa)}^2 \leq (2 \pi)^{-1} \int_0^{2\pi} \abs{w_\gamma(x,\kappa+\varrho e^{i\phi})}^2d \phi$, from the Cauchy-Schwarz inequality. Since the above estimate is valid uniformly in $\varrho \in (0,\varepsilon)$, we find that
\beas
\abs{w_\gamma(x,\kappa)}^2 & \leq & \frac{1}{2 \pi \varepsilon} \int_0^{\varepsilon }\int_0^{2\pi} \abs{w_\gamma(x,\kappa+\varrho
e^{i\phi})}^2d\phi d\varrho \\
& \leq & \frac{1}{2 \pi \varepsilon} \int_{\abs{\tau}\leq \varepsilon} \int_{\abs{t-\kappa}\leq \varepsilon } \abs{w_\gamma(x,t+i\tau)}^2 dt d\tau \\
& \leq & \frac{1}{2 \pi \varepsilon} \int_{-T_0}^{T_0} \| w_{\gamma,t}(x,\cdot) \|_{L^2(-\varepsilon, \varepsilon)}^2 dt,
\eeas
and hence
$$
\| w_\gamma \|_{L^2( (\Omega_3 \setminus \Omega_2) \times (-T_0 \slash 2 , T_0 \slash 2) )}^2
\leq
C \int_{-T_0}^{T_0} \| w_{\gamma,t} \|_{L^2(\left( \Omega_3 \setminus \Omega_2 \right) \times (-\varepsilon , \varepsilon) )}^2 dt,
$$
upon integrating with respect to $(x,\kappa)$ over $\left( \Omega_3 \setminus \Omega_2 \right) \times (-T_0 \slash 2,T_0 \slash 2)$.
Thus, bearing in mind \eqref{b2}-\eqref{b3} and applying Lemma \ref{L4.4}, we end up getting that
\bel{b6}
\| w_\gamma \|_{L^2( (\Omega_3 \setminus \Omega_2) \times (-T_0 \slash 2 , T_0 \slash 2) )}^2
\leq C \left( e^{-C_3\gamma}
+e^{C_4\gamma} \| \p_\nu \widetilde{w} \|_{L^2(\Gamma_* \times (-3T_0 , 3T_0))}^2 \right).
\ee
Finally, we notice upon arguing in the same way, that $\nablas w_\gamma$ may be substituted for $w_\gamma$ in the left hand side of \eqref{b6} so the desired result follows from this
and \eqref{b6}.
\end{proof}

We next establish the coming result with the help of Lemma \ref{L4.5}.

\begin{lemma}
\label{L4.6}
For any $T>0$ fixed, we may find $T_1 \in (0,T)$ such that we have
\bel{eqv}
 \| v \|_{L^2( ( \Omega_3 \setminus \Omega_2 ) \times (-T_1,T_1) )}^2 +  \| \nablas v \|_{L^2( ( \Omega_3 \setminus \Omega_2 ) \times (-T_1,T_1) )}^2
\leq C \left( \frac{1}{\gamma^{2\mu N}} + e^{\delta_5 \gamma} \| \p_\nu v \|_{L^2( \Gamma_* \times (-T,T) )}^2 \right),
\ee
for every $\gamma \in [\gamma_0,+\infty)$. Here $C>0$ depends only on $\omega$, $\omega_0$, and $T$, and the constant $\delta_5$ is the same as in Lemma \ref{L4.5}.
\end{lemma}

\begin{proof}
Let $\widehat{u}(\cdot,\zeta)$, for $\zeta \in \R$, denote the partial Fourier transform computed at $\zeta$ of $t \mapsto u(\cdot,t)$. In light of \eqref{b5}, it holds true that
\bel{4.41}
\widehat{\theta \widetilde{w}}(\cdot,\zeta)-\widehat{w_\gamma}(\cdot,\zeta)
= (1-\widehat{F_\gamma})\widehat{\theta \widetilde{w}}(\cdot,\zeta),\ \zeta \in \R.
\ee
Therefore, taking into account that $\widehat{F_\gamma}(\zeta)=e^{-\para{\zeta \slash \gamma^\rho}^{2m}}$, using that
$1-e^{-y^{2m}} \leq C y^{N}$ for all $y \in [0,+\infty)$ (since $2m \geq N$ from \eqref{def-m}) and recalling that $\rho > \mu$ and $\gamma >1$, we derive from \eqref{4.41} that
\bel{4.42}
\| \widehat{\theta \widetilde{w}}(x,\cdot)-\widehat{w_\gamma}(x,\cdot) \|_{L^2(\R)} 
\leq \frac{1}{\gamma^{\mu N}} \|\zeta^{N} \widehat{\theta \widetilde{w}}(x,\cdot) \|_{L^2(\R)},\ x \in \Omega_3 \setminus \Omega_2.
\ee
Since the function $\theta$, defined in \eqref{b2}, is supported in $(-3 T_0, 3 T_0)$, it then follows from \eqref{4.42}
that
\bel{4.43}
\| \theta \widetilde{w}(x,\cdot) - w_\gamma(x,\cdot) \|_{L^2(\R)} \leq \frac{1}{\gamma^{\mu N}} \| \p_t^{N} (\theta \widetilde{w})(x,\cdot) \|_{L^2(\R)} 
\leq \frac{C}{\gamma^{\mu N}} \| \widetilde{w}(x,\cdot) \|_{H^{N}(-3T_0,3T_0)},\ x \in \Omega_3 \setminus \Omega_2,
\ee
for some constant $C>0$ depending neither on $x$ nor on $\gamma$.
Thus, bearing in mind that $\theta(t)=1$ for each $t \in [-T_0 \slash 2, T_0 \slash 2]$, we get upon squaring and integrating both sides of
\eqref{4.43} with respect to $x$ over $\Omega_3 \setminus \Omega_2$, that
\bel{4.44}
\| \widetilde{w} - w_\gamma \|_{L^2( (\Omega_3 \setminus \Omega_2 ) \times (-T_0/2 , T_0/2) )} \leq \frac{C}{\gamma^{\mu N}},\ \gamma \in (1,+\infty).
\ee
Here we used \eqref{4.14} to bound from above the $L^2( \Omega_3 \setminus \Omega_2, H^{N}(-3T_0,3T_0))$-norm of the function
$\widetilde{w}(x,t)=w(x, T \slash (3 T_0) t)$, uniformly in $\gamma \in (0,1)$.

Further, we proceed in the same way as in the derivation of \eqref{4.44} and obtain that
$$\| \nablas \widetilde{w} - \nablas w_\gamma \|_{L^2( (\Omega_3 \setminus \Omega_2 ) \times (-T_0/2 , T_0/2) )} \leq  \frac{C}{\gamma^{\mu N}}.$$
From this, \eqref{4.44} and Lemma \ref{L4.5}, it follows for all $\gamma \in [\gamma_0,+\infty)$ that
\beas
& & \| \widetilde{w} \|_{L^2( (\Omega_3 \setminus \Omega_2 ) \times (-T_0/2 , T_0/2) )}^2 + \| \nablas \widetilde{w} \|_{L^2( (\Omega_3 \setminus \Omega_2 ) \times (-T_0/2 , T_0/2) )}^2 \\
& \leq & \frac{C}{\gamma^{2\mu N}}+  \| w_\gamma \|_{L^2( (\Omega_3 \setminus \Omega_2 ) \times (-T_0/2 , T_0/2) )}^2 + \| \nablas w_\gamma \|_{L^2( (\Omega_3 \setminus \Omega_2 ) \times (-T_0/2 , T_0/2) )}^2 \\
& \leq & C \left( \frac{1}{\gamma^{2\mu N}}+  e^{\delta_5 \gamma} \| \p_\nu \widetilde{w} \|_{L^2(\Gamma_* \times (-3T_0 , 3T_0) )}^2 \right),
\eeas
provided $T_0$ is sufficiently large. 
As a consequence we have
\bea
& & \| w \|_{L^2( ( \Omega_2 \setminus \Omega_3 ) \times (-hT_0/2 , hT_0/2) )}^2 + \| \nablas w \|_{L^2( ( \Omega_2 \setminus \Omega_3 ) \times (-hT_0/2 , hT_0/2) )}^2  \nonumber \\
& \leq & C \left( \frac{1}{\gamma^{2\mu N}} + e^{\delta_5\gamma} \| \p_\nu w\|_{L^2(\Gamma_* \times (-3hT_0 , 3hT_0) )}^2 \right), \label{4.45}
\eea
since $\tilde{w}(\cdot,t)=w(\cdot,h t)$ for every $t \in \R$. Finally, bearing in mind that $h=T \slash (3T_0)$ and recalling from \eqref{4.12} that $w(x,t)=\chi(x') v(x,t)=v(x,t)$ for every $(x,t) \in (\Omega_3 \setminus \Omega_2) \times (-T,T)$, we end up getting from \eqref{4.45} that
$$
\| v \|_{L^2( ( \Omega_2 \setminus \Omega_3 ) \times (-T \slash 6 , T \slash 6) )}^2 + \| \nablas v \|_{L^2( ( \Omega_2 \setminus \Omega_3 ) \times (-T \slash 6, T \slash 6) )}^2 
\leq C \left( \frac{1}{\gamma^{2\mu N}}+e^{\delta_5\gamma} \| \p_\nu v\|_{L^2(\Gamma_* \times ( -T , T) )}^2 \right).
$$
This yields \eqref{eqv} with $T_1=T \slash 6$.
\end{proof}
Armed with Lemma \ref{L4.6}, we are now in position to complete the proof of Lemma \ref{L2.1}. This can be done upon applying \eqref{eqv} with
$\gamma = \abs{\log \norm{\p_\nu v}_*}  \slash \delta_5$, which is allowed when $\norm{\p_\nu v}_* \in (0, e^{-\delta_5 \gamma_0}]$ so that we have $\gamma \geq \gamma_0$. We find that
\bea
& & \| v \|_{L^2( ( \Omega_2 \setminus \Omega_3 ) \times (-T_1,T_1) )}^2 + \| \nablas v \|_{L^2( ( \Omega_2 \setminus \Omega_3 ) \times (-T_1, T_1) )}^2 \nonumber \\
& \leq & C \left( \abs{\log \norm{\p_\nu v}_*}^{-2\mu N} + \norm{\p_\nu v}_* \right) \leq C' \abs{\log \norm{\p_\nu v}_*}^{-2\mu N}, \label{4.50}
\eea
for $C'$ is a suitable positive constant. On the other hand, when $\norm{\p_\nu v}_* > e^{-\delta_5 \gamma_0}$, it is clear from the estimate $\| v \|_{L^2(-T,T;H^1(\Omega))} \leq 2 C \| u_0 \|_{H^{2(N+1)}(\Omega)}$, arising from \eqref{4.3}, that
$$
\| v \|_{L^2( ( \Omega_2 \setminus \Omega_3 ) \times (-T_1 , T_1) )}^2 + \| \nablas v \|_{L^2( ( \Omega_2 \setminus \Omega_3 ) \times (-T_1, T_1) )}^2 \leq C \norm{\p_\nu v}_*^{2 \mu N},
$$
so we get \eqref{2.11} directly from this and from \eqref{4.50} (upon enlarging $T$ into $6T$ so that we have $T_1=T \slash 6=T$).


\section{Appendix}
\setcounter{equation}{0} 
In this appendix we prove the parabolic Carleman estimate stated in Lemma \ref{L4.2}. Incidentally we make precise the dependence with respect to $\lambda$ and 
$h$, of the constant $C_0$ appearing in the right hand side of \eqref{4.11}.

We stick with the notations of Section 4 and start by gathering several useful straightforward properties of the weight functions $\varphi_0$ and $\alpha$, defined by \eqref{4.8}-\eqref{4.9}, in the coming lemma.
 
\begin{lemma}
\label{LA.1}
We may find three constants $\lambda_0 \geq 1$, $c > 0$ and $c' >0$, all of them depending only on $\omega_0$, such that for each $\lambda \geq \lambda_0$ and
all $(x',\tau) \in \omega_0 \times (-1,1)$, the following estimates hold simultaneously:
\bea
D_{x'}^2\alpha( \nablas\alpha , \nablas \alpha ) (x',\tau)& \geq & c \lambda^4 \varphi_0(x',\tau)^3, \label{A.1} \\
\abs{\Deltas \alpha(x',\tau)} & \leq & c'\lambda^2 \varphi_0(x',\tau)^2, \label{A.2}\\
\abs{\Deltas^2 \alpha(x',\tau)} & \leq & c'\lambda^4 \varphi_0(x',\tau)^3, \label{A.2b}\\
\abs{(\p_\tau \alpha) (\Deltas \alpha)(x',\tau)} & \leq & c' \lambda^2 \varphi_0(x',\tau)^3, \label{A.3}\\
\abs{\p_\tau^2 \alpha(x',\tau)} & \leq & c' \varphi_0(x',\tau)^3, \label{A.4}\\
\abs{\nablas(\p_\tau\alpha)(x',\tau) \cdot \nablas \alpha(x',\tau)} & \leq & c' \lambda^2 \varphi_0(x',\tau)^3, \label{A.5} \\
\Ds^2\alpha(\xi',\xi')(x',\tau) & \geq & -c \lambda \varphi_0 \abs{\xi'}^2,\ \xi' \in\R^{n-1}.\label{A.6}
\eea
\end{lemma}
In the sequel $C$ denotes a generic positive constant which depends only on $\omega_0$, whose value can change from line to line.

Put  $z:=e^{\s\alpha} w$ and notice for further reference from \eqref{4.8}-\eqref{4.9} that
\bel{A.9}
z(x,\pm 1)=0,\ x \in \Omega_0=\omega_0\times\R.
\ee 
Next, setting $f_\s :=e^{\s\alpha}(h^{-1}\p_\tau-\Delta)w$, we find through direct computation that
\bel{A.10}
L_1 z+L_2 z =g_\s := f_\s -\s (\Deltas \alpha) z,
\ee
where
\bea
L_1 z & := & h^{-1}\p_\tau z+2\s\nablas\alpha\cdot\nablas z, \label{A.11}\\
L_2 z & := & -\Delta z -\s \left( h^{-1}(\p_\tau \alpha)   + \s \abs{\nablas\alpha}^2 \right)z. \label{A.12}
\eea 
Due to \eqref{A.10}, we have the identity
\bel{A.14}
\sum_{j=1,2} \| L_j z \|_{L^2(\Omega_0 \times (-1,1))}^2 
+2 \Pre{\int_{-1}^1 \int_{\Omega_0} L_1z \overline{L_2z} dx d\tau}= \| g_\s \|_{L^2(\Omega_0 \times (-1,1))}^2,
\ee
so we are left with the task of estimating the $L^2(\Omega_0 \times (-1,1))$-scalar product of $L_1z$ and $L_2z$, appearing in the left hand side of \eqref{A.14}. In view of \eqref{A.11}-\eqref{A.12}, we have
\bel{A.15}
2 \int_{-1}^1\int_{\Omega_0} L_1 z \overline{L_2z} dx d\tau= \sum_{j=1,2,3} I_j,
\ee
where
\bea
I_1 & := & 2h^{-1}\int_{-1}^1 \int_{\Omega_0} \p_\tau z \para{-\Delta \overline{z} - \s \left(
h^{-1}(\p_\tau\alpha) + \s \abs{\nablas\alpha}^2 \right) \overline{z}} dx d\tau,\label{A.16}\\
I_2 & := & -4 \s \int_{-1}^1\int_{\Omega_0} \left( \nablas \alpha \cdot \nablas z \right) \Delta \overline{z} dx d\tau,\label{A.17}\\
I_3 & := & -4 \s^2 \int_{-1}^1 \int_{\Omega_0} \left( \nablas \alpha \cdot \nablas z \right) \para{h^{-1}(\p_\tau\alpha) + \s \abs{\nablas\alpha}^2} \overline{z} dx d\tau.\label{A.18}
\eea
Bearing in mind that $z_{|\p\Omega_0\times(-1,1)}=0$ from the very definition of $z$ (since the same is true for the function $w$, according to \eqref{4.13}), we integrate by parts with respect to $x$ in \eqref{A.16}and obtain that
$$
\Pre{I_1}=h^{-1}\int_{-1}^1 \int_{\Omega_0}\para{\p_\tau \abs{\nabla
z}^2-\s (h^{-1}\p_\tau \alpha + \s \abs{\nablas \alpha}^2)\p_\tau \abs{z}^2} dx d\tau.
$$
Recalling \eqref{A.9}, we next integrate by parts with respect to $\tau$ and find that
\bel{A.19}
\Pre{I_1}=\s h^{-1}\int_{-1}^1\int_{\Omega_0}\para{ h^{-1} (\p_\tau^2 \alpha) \abs{z}^2 + 2\s \nabla_{x'} (\p_\tau\alpha)
\cdot\nablas\alpha} |z|^2dx d\tau.
\ee
The second term, $I_2$, is handled in a similar way. Namely, we integrate by parts with respect to $x$ in the right hand side of \eqref{A.17}, use the identity $\nabla z = (\pd_\nu z) \nu$ on $\pd \Omega_0 \times (-1,1)$, and get
\beas
\Pre{I_2} & = & 4 \s\int_{-1}^1 \int_{\Omega_0} \Ds^2 \alpha(\nablas z,\nablas \overline{z}) dx d\tau + 2 \s \int_{-1}^1\int_{\Omega_0} \nabla \alpha \cdot \nabla | \nabla z |^2 dx d\tau \\
& &  - 4 \s \int_{-1}^1 \int_{\p\Omega_0} ( \p_\nu\alpha) \abs{\p_\nu
z}^2 dx d\tau.
\eeas
Therefore, taking into account that
$$ \int_{-1}^1\int_{\Omega_0} \nabla \alpha \cdot \nabla | \nabla z |^2  dx d\tau = -\int_{-1}^1\int_{\Omega_0} (\Delta \alpha) | \nabla z |^2 dx d\tau
+ \int_{-1}^1 \int_{\p\Omega_0} (\p_\nu \alpha) \abs{\p_\nu z}^2 dx d\tau,
$$
we see that
\bea
\Pre{I_2} & = & 4 \s\int_{-1}^1 \int_{\Omega_0} \Ds^2 \alpha(\nablas z,\nablas \overline{z})dx d\tau  - 2 \s \int_{-1}^1 \int_{\Omega_0} (\Delta \alpha) | \nabla z |^2  dx d\tau \nonumber \\
& & - 2 \s \int_{-1}^1 \int_{\p \Omega_0} (\p_\nu \alpha) \abs{\p_\nu z}^2 dxd\tau. \label{A.20}
\eea
Let us now compute the real part of $I_3$. To this end we notice upon integrating by parts with respect to $x$ that
\bea
& & 2 \Pre{\int_{-1}^1 \int_{\Omega_0} \abs{\nablas\alpha}^2 (\nablas\alpha \cdot\nablas
z) \overline{z} dxd\tau} \nonumber \\
& = & - \int_{-1}^1\int_{\Omega_0} \left( \abs{\nablas
\alpha}^2 \Deltas\alpha + 2 \Ds^2\alpha(\nablas\alpha,\nablas\alpha) \right) |z|^2 dx d\tau, \label{A.21}
\eea
and that
\beas
& & 2 \Pre{\int_{-1}^1 \int_{\Omega_0} (\p_\tau\alpha) (\nablas\alpha \cdot \nablas z) \overline{z} dx d\tau} \nonumber \\
& = & - \int_{-1}^1\int_{\Omega_0} \left( (\p_\tau\alpha) (\Deltas\alpha) + (\nablas(\p_\tau\alpha) \cdot\nablas\alpha) \right) \abs{z}^2 dx d\tau.
\eeas
It follows from this, \eqref{A.18} and \eqref{A.21} that
\bea
\Pre{I_3} & = & 2\s^3\int_{-1}^1\int_{\Omega_0} \left( \abs{\nablas
\alpha}^2 (\Deltas\alpha) +
2 \Ds^2\alpha(\nablas\alpha,\nablas\alpha) \right) \abs{z}^2 dxd\tau \nonumber \\
& & + 2 \s^2 h^{-1}\int_{-1}^1\int_{\Omega_0} \left( (\p_\tau\alpha)(\Deltas\alpha) + \nablas(\p_\tau\alpha) \cdot \nablas\alpha \right)
\abs{z}^2 dxd\tau. \label{A.23}
\eea
Finally, putting \eqref{A.15}, \eqref{A.19}, \eqref{A.20} and \eqref{A.23} together we end up getting that
\bea 
2 \Pre{\int_{-1}^1\int_{\Omega_0} L_1 z \overline{L_2 z} dx d\tau} & = & -2 \s \int_{-1}^1\int_{\Omega_0} (\Deltas\alpha) \para{ \abs{\nabla z}^2
 -\s^2 \abs{\nablas \alpha}^2 \abs{z}^2} dx d\tau \nonumber \\
& & + 4 \s \int_{-1}^1\int_{\Omega_0} \Ds^2 \alpha(\nablas z,\nablas \overline{z})  dx d\tau + \sum_{j=1,2,3} J_j, \label{A.24}
\eea
where have set
\bea
J_1 &:=& 4\s^3\int_{-1}^1 \int_{\Omega_0} \Ds^2\alpha(\nablas\alpha,\nablas\alpha) \abs{z}^2 dx d\tau, \label{A.25}\\
J_2 &:=& -2 \s \int_{-1}^1 \int_{\p\Omega_0} (\p_\nu \alpha) \abs{\p_\nu z}^2 dx d\tau, \label{A.26} \\
J_3 &:=& 2\s^2h^{-1}\int_{-1}^1\int_{\Omega_0} \para{2 \nablas (\p_\tau \alpha) \cdot \nablas \alpha + (\p_\tau\alpha) \Deltas\alpha}
 \abs{z}^2 dx d\tau \nonumber \\
& & +\s h^{-2}\int_{-1}^1\int_{\Omega_0} (\p_\tau^2 \alpha) \abs{z}^2 dx d\tau.\label{A.27}
\eea
The next step of the proof is to bound from below each of the three terms $J_j$, $j=1,2,3$, appearing in the right hand side of \eqref{A.24}.
In view of \eqref{A.1}, $J_1$ is easily treated by \eqref{A.25}, as we have
\bel{A.28}
J_1 \geq 4 c \s^3 \lambda^4 \| \varphi_0^{3\slash 2} z \|_{L^2(\Omega_0 \times (-1,1))}^2,\ \lambda \geq\lambda_0,
\ee
where $c$ is the constant defined in Lemma \ref{LA.1}.
Similarly, we deduce from \eqref{A.3}--\eqref{A.5} and \eqref{A.27} that
$$
\abs{J_3} \leq (6\s^{2}h^{-1} +\s
h^{-2}) C_2 \lambda^2 \| \varphi_0^{3\slash 2} z \|_{L^2(\Omega_0 \times (-1,1))}^2,\ \lambda \geq \lambda_0,
$$
where $c'$ is the same as in Lemma \ref{LA.1}.
Therefore, for all $\s\geq \s_0 \slash h$,it follows readily from this and \eqref{A.28} that
\bel{A.30}
J_1+J_3\geq C \s^3\lambda^4 \| \varphi_0^{3\slash 2} z \|_{L^2(\Omega_0 \times (-1,1))}^2,\ \lambda \geq \lambda_0,
\ee
for some constant $C>0$, depending only on $\omega_0$.
On the other hand, since
$\p_\nu\alpha=\lambda (\p_\nu \psi_0) \varphi_0$ on $\pd \Omega_0 \times (-1,1)$ , by \eqref{4.8}, the identity \eqref{A.26} yields
$$
J_2 \geq -2 \s\lambda  \| \varphi_0^{1 \slash 2} (\p_\nu \psi_0)^{1 \slash 2} \p_\nu z \|_{L^2(\Gamma_* \times (-1,1)}^2.
$$
Now, putting this together with \eqref{A.6}, \eqref{A.24} and \eqref{A.30}, we obtain for all $\lambda \geq \lambda_0$ and all $\sigma \geq \sigma_0 \slash h$ that
\bea
2 \Pre{\int_{-1}^1\int_{\Omega_0} L_1z \overline{L_2z} dx d\tau} 
& \geq & C \s^3 \lambda^4 \| \varphi_0^{3\slash 2} z \|_{L^2(\Omega_0 \times (-1,1))}^2 - 4 c \s \lambda \| \varphi_0^{1 \slash 2} \nablas
z \|_{L^2(\Omega_0 \times (-1,1))} \nonumber \\
& & -2 \s \lambda \| \varphi_0^{1 \slash 2} (\p_\nu \psi_0)^{1 \slash 2} \p_\nu z \|_{L^2(\Gamma_* \times (-1,1))}^2 -2S, \label{A.31}
\eea
where
\bel{A.32}
S:=\s \int_{-1}^1\int_{\Omega_0} (\Deltas\alpha) \para{\abs{\nabla z}^2
 -\s^2 \abs{\nablas \alpha}^2\abs{z}^2} dxd\tau.
\ee
The rest of the proof involves bounding $S$ from above. To this purpose we recall from \eqref{A.12} that
$$
\abs{\nabla z}^2 - \s^2\abs{\nablas \alpha}^2\abs{z}^2 = \Pre{(L_2 z) \overline{z}} + \frac{\Delta |z|^2}{2} + \s h^{-1} (\p_\tau \alpha) | z |^2,
$$ 
and then deduce from \eqref{A.32} that
\beas
S &=& \s\int_{-1}^1\int_{\Omega_0}(\Deltas\alpha) \left( \Delta | z |^2 \slash 2 + \Pre{(L_2 z) \overline{z}} + \s h^{-1} (\p_\tau\alpha) |z|^2 \right)  dx  d \tau \\
&=& \s \int_{-1}^1 \int_{\Omega_0} \left( (\Deltas^2\alpha \slash 2) \abs{z}^2 + (\Deltas\alpha) \para{\Pre{(L_2z) \overline{z}} + \s h^{-1}(\p_\tau\alpha) |z|^2} \right) dx d\tau.
\eeas
As a consequence we have
\beas
\abs{S} & \leq & \norm{L_2z}^2_{L^2(Q)} \slash 4 + \| |\Deltas^2 \alpha|^{1 \slash 2} z \|_{L^2(\Omega_0 \times (-1,1))}^2 \\
& & + \s^2 \left( \| (\Deltas \alpha) z \|_{L^2(\Omega_0 \times (-1,1))}^2 + h^{-1} \| \abs{\p_\tau\alpha} z \|_{L^2(\Omega_0 \times (-1,1))}^2 \right).
\eeas
This, together with \eqref{A.2}-\eqref{A.2b} and \eqref{A.5}, yields that
$$
\abs{S} \leq \norm{L_2z}^2_{L^2(Q)} \slash 4 +C \left( \s \lambda^4 + \s^2 \lambda^2 (\lambda^2+h^{-1}) \right)
\| \varphi_0^{3 \slash 2} z \|_{L^2(\Omega_0 \times (-1,1))}^2.
$$
It follows from this and \eqref{A.31} upon taking $\s \geq \s_0 \slash h$ that
\beas
2 \Pre{ \int_{-1}^1 \int_{\Omega_0} L_1z \overline{L_2z} dx d\tau}+ \norm{L_2z}^2_{L^2(Q)} \slash 2 & \geq & 
\s^3 \lambda^4 \| \varphi_0^{3 \slash 2} z \|_{L^2(\Omega_0 \times (-1,1))}^2 - 4 c \s \| \varphi_0^{1 \slash 2} \nablas z \|_{L^2(\Omega_0 \times (-1,1))}^2 \\
& & -2 \s \lambda \| \varphi_0^{1 \slash 2} (\p_\nu \psi_0)^{1 \slash 2} \p_\nu z \|_{L^2(\Gamma_* \times (-1,1))}^2.
\eeas
Having estimated all the contributions depending on $h$, we proceed as in \cite[Appendix]{[FZ]} and obtain the desired result.



\end{document}